\documentclass[a4paper,12pt]{article}
\usepackage{amssymb,amsmath,amsthm}
\usepackage{graphicx}

\usepackage{amsfonts}
\usepackage{latexsym}
\usepackage{color}
\usepackage[english]{babel}
\usepackage{hyperref}

\usepackage{hyperref}
\usepackage{verbatim} 
\usepackage{bm} 
\usepackage{xcolor}
\usepackage{graphicx} 
\usepackage{subcaption}
\usepackage[font=footnotesize,labelfont=bf]{caption}
\usepackage{mathrsfs}
\definecolor{darkgr}{rgb}{0.0, 0.62, 0.42}

\numberwithin{equation}{section}

\newcommand{\Span}{\mathrm{span}}

\newtheorem{lemma}{Lemma}[section]
\newtheorem{theorem}[lemma]{Theorem}
\newtheorem{proposition}[lemma]{Proposition}

\newtheorem{remark}[lemma]{Remark}

\newtheorem{definition}[lemma]{Definition}

\newcommand{\csi}{\xi}

\newcommand{\calc}{{\cal C}}

\def\mass{\mathcal M \kern -2pt a \kern-2pt d}

\def\im{{\rm i}}

\def\cI{{\mathcal{I}}}

\def\cC{{\mathcal{C}}}

\def\cF{{\mathcal{F}}}

\def\cE{{\mathcal{E}}}
\def\cW{{\mathcal{W}}}

\def\cU{{\mathcal{U}}}

\def\cB{{\mathcal{B}}}

\newcommand{\ep}{\epsilon}
\newcommand{\varep}{\varepsilon}

\def\scal#1#2{\left(#1;#2\right)}

\newcommand{\R}{\mathbb R}
\newcommand{\C}{\mathbb C}
\newcommand{\Z}{\mathbb Z}
\newcommand{\N}{\mathbb N}
\newcommand{\T}{\mathbb T}

\newcommand{\Op}{{Op}^{W}}
\newcommand{\OPS}{{OPS}_{\delta}}

\def\green#1{\relax}

\def\red#1{{#1}}
\def\blu#1{{#1}}

\newcommand{\nform}{z}
\newcommand{\NForm}{Z}

\newcommand{\sobop}[3]{\left\| {#1}\right\|_{ {#2}, {#3}} }

\newcommand{\res}{(\textrm{res})}
\newcommand{\nr}{(\textrm{nr}
	)}

\newcommand{\partialt}{\partial_{t}}

\newcommand{\id}{\textrm{Id}}

\newcommand{\cinfty}{{\cal C}^{\infty}}

\def\be{{\bf e}}

\newcommand{\mort}{{\cal I}_{M}}
\newcommand{\timereg}{\cinfty_b}
\newcommand{\simb}{\textrm{S}_\delta}

\def\mper{M^{\perp}}
\newcommand{\csiort}{{\csi_{M^\bot}}}
\def\growth{\cite{growth}}
\def\stimevecchie{\cite{How89, How92, joy,nen,barjoy, joy2,duclos}}

\def\ac{{\mathcal A}\kern-.7pt\ell\kern-.9pt\mathcal{S}}

\begin{document}

\title{Growth of Sobolev norms for
  \red{unbounded perturbations of the Schr\"odinger equation} on flat tori}
\date{}

\author{ Dario Bambusi\footnote{Dipartimento di Matematica, Universit\`a degli Studi di Milano, Via Saldini 50, I-20133
Milano. 
 \textit{Email: } \texttt{dario.bambusi@unimi.it}}, Beatrice Langella\footnote{International School for Advanced Studies (SISSA), via Bonomea 265, I-34136 Trieste.
 \textit{Email: } \texttt{beatrice.langella@sissa.it}}, Riccardo Montalto \footnote{Dipartimento di Matematica, Universit\`a degli Studi di Milano, Via Saldini 50, I-20133
Milano.
 \textit{Email: } \texttt{riccardo.montalto@unimi.it}}
}

\maketitle

\begin{abstract} We prove a $\langle t\rangle^\varepsilon$ \red{upper} bound on the
  growth of Sobolev norms \blu{for all solutions of Schr\"odinger
    equations  on flat tori with a Hamiltonian which is an unbounded time dependent perturbation of the
      Laplacian.}
\end{abstract} \noindent


\medskip




\section{Introduction}\label{intro}

We consider the Schr\"odinger equation
\begin{equation}\label{main.general}
\im \partialt \psi = H{(t)}\psi\,, \quad H{(t)}= -\Delta + V(t) \quad \quad \psi
\in L^2(\T^d_\Gamma)\,,
\end{equation}
where $V(t)$ is a smooth family of time dependent self-adjoint unbounded pseudo-differential operators of order $m <2$ and $\T^d_\Gamma =
\R^d/\Gamma$ is a torus with arbitrary periodicity lattice
$\Gamma$. We prove a $ \langle t\rangle^\varepsilon$ \blu{upper} bound on the growth
of the Sobolev norms of the solution. As far as we know this is the first
result of this kind for \emph{unbounded perturbations of the Laplacian
  on a manifold different from a Zoll manifold}.

We recall that previous results ensuring $\langle
t\rangle^\varepsilon$ \blu{bounds on the}
growth of Sobolev norms for time dependent Schr\"odinger equations on tori were proved in \cite{Bou,Delort,BM19} for \emph{bounded}
perturbations of the Laplacian and they were based on the use of a
lemma by Bourgain on the structure of the ``resonant clusters'' of a
suitable lattice.

Our proof follows the general strategy of
    \growth\ and consists of three steps.  In the first
step we use a normal form approach to conjugate \eqref{main.general}
to an equation which is a smoothing perturbation of a suitable normal
form equation. The second step consists of the analysis of the
dynamics generated by the normal form operator, and in the third step
we add the smoothing perturbation and obtain our main result. 

The main technical novelty of the present paper is that the first and the second
parts of the proof are based on a local formulation of normal form
theory which was developed in \cite{risonante} in order to deal with
spectral problems.  Actually the present paper originates from the
understanding that the main structural result of \cite{risonante} can
be used in order to deal with the problem of growth of Sobolev
norms. Indeed the structural theorem of \cite{risonante} provides
a decomposition of the Fourier space in dyadic blocks which are invariant for
operators in normal form\footnote{according to Definition
  \ref{def.nor} below} and these blocks are similar to the one
introduced by Bourgain to deal with the case of bounded
perturbations. The main point of our construction is that, since it is
based on pseudo-differential calculus, one can control the
unboundedness of the perturbation exploiting the fact that the
commutator of pseudo-differential operators is more regular than the
original factors.  The result of this part of the proof is a
Lemma ensuring that, in the dynamics of the normal form equations, the
Sobolev norms of the solutions are bounded uniformly in time.  Finally
the third step of the proof consists in adding the remainder, using
Duhamel formula and a standard interpolation result to get the
$\langle t\rangle^\varepsilon$ bound.

We conclude this introduction by recalling that the problem of
understanding the flow of energy to high frequency modes in time
dependent linear equations has a long history and has been
investigated mainly in the context of time periodic or quasiperiodic
perturbations of harmonic or anharmonic quantum oscillators. We first
mention the works \stimevecchie\ in which perturbations of 1-d
oscillators were studied and estimates on the transfer of energy
between low and high frequency modes were given. Then, in the works
\cite{BG01,2, 1,3} the related problem of reducibility {for harmonic oscillators} was studied in
the 1-d case with unbounded perturbations. For the higher dimensional
case only a few reducibility results are known:  the first
breakthrough result has been proved by Eliasson and Kuksin
\cite{EK09}, in the case of the Schr\"odinger equation on $\T^d$ with
a bounded analytic time quasi-periodic potential; in the case of
unbounded potentials, finally we mention \cite{BGMR1, BLM19, Mon17, M,
    GrebFeola,growth} (see also \cite{MaRo}) in which the strategy of proof is similar to
  ours, but the situation is much simpler, since the geography of the
resonances is trivial in all the models considered therein.

We finally recall that in some cases logarithmic bounds on the growth of
Sobolev norms have been obtained \cite{Bou2,wang08,fang}.  We think that this kind
of results can be obtained by developing the methods of the present
paper and adding quantitative estimates for the analytic or Gevrey case.\\

\noindent
{\it Acknowledgments.} We warmly thank Alberto Maspero who convinced
us to try to apply the methods of \cite{risonante} to the problem of
growth of Sobolev norms in an unbounded context: his encouragement was
what convinced us to tackle this problem and finally led to this paper. {We also thank Alberto Maspero and Massimiliano Berti for interesting discussions on the structure of the invariant blocks.}

This work was partially supported by GNFM.

	\section{Main result}
	Let ${\bf e}_1, {\bf e}_2,
	\ldots, {\bf e}_d$ be a basis of $\Gamma,$ namely 
	\begin{equation}\label{definizione Gamma}
	\Gamma := \Big\{  \sum_{i = 1}^d k_i {\bf e}_i : k_1, \ldots, k_d \in
	\Z\Big\}\ .
	\end{equation} 
	By introducing in $\R^d$ the basis of the vectors
	$\be_i$, one is reduced to an equation of the form
	\begin{align}
	\label{g}
	\im \partialt \psi = \check H (t)\psi =\left(-\Delta_{g}+ \check V(t)\right) \psi\ ,
	\\
	\label{deltag}
	-\Delta_{g}:=-\sum_{A, B = 1}^{d} g^{AB} \partial_A \partial_B\,,
	\end{align}
	with p.b.c. on the standard torus $\T^d:=\R^d/(2\pi \Z)^d$.
	Here $\{g^{A B}\}_{A, B = 1}^d$ is the inverse of the matrix with elements
	\begin{equation}
	\label{g.1}
	 g_{AB}:=\be_A\cdot\be_B\,,
	\end{equation}
	\red{and $\check V(t)$ is a family of pseudodifferential operators on $L^2(\T^d)$ of the same order as $V(t)$\footnote{\red{actually $\forall t \in \R$, if $v(t)$ is the symbol of $V(t)$ according to Definition \ref{def.op} below, then $\check{V}(t)$ has symbol $\check{v}(t)$, where $\check{v}(t)$ coincides with $v(t)$ written in the new variables.}} }
	From now on, we restrict our
	analysis to the equation \eqref{g} on $L^2(\T^d)$, and in the following we will omit the checks. Thus we consider a Schr\"odinger equation of the form
        \begin{align}
        \label{main.00}
	\im \partialt \psi&= H(t) \psi \ , \quad \psi \in L^2(\T^d)\,,
\\
  \label{main.eq}
H(t)&= -\Delta_{g} +
        V(t) \ .
                \end{align}
	In the following we will only deal with
	scalar products and norms with respect to the metric $g$; we will
	denote by
	\begin{equation}
	\label{scal}
	\scal \xi \eta:= \sum_{A, B = 1}^{d} g^{AB}\xi_A\eta_B\ ,\quad 	\| \xi\|^2:= \scal \xi\xi\ 
	\end{equation}
	the scalar product and the norm of covectors with respect
	to this metric.
	Finally, let $d\mu_g(x)$ be the volume form corresponding to
	$g$: for any $\psi \in L^2(\T^d),$ we define its Fourier
        coefficients by
	$$
	\hat{\psi}_k = \frac{1}{\red{\mu_g(\T^d)}} \int_{\T^d} \psi(x) e^{-\im k \cdot x}\ d \mu_g(x)\,, \quad \forall k \in \Z^d\,,
	$$
	where $k \cdot x = \sum_{A=1}^{d} k_A x^{A}$.
    
	We now define precisely the class of symbols and of pseudo-differential operators we will use:
	\begin{definition}\label{def.simb}
Let $0<\delta\leq 1$, and let $g$ be a flat Riemannian metric. Given $m \in \R,$ we say that $v \in \cinfty \left(\T^d \times \R^d ;\C \right)$ is a symbol of order $m,$ and we write $v \in {\simb^m,}$ if for any $N_1, N_2 \in \N$ there exists a positive constant $C_{N_1, N_2} = C_{N_1, N_2}(v)$ such that
	\begin{equation}\label{simb}
	\sup_{x \in \T^d} \left| \partial^{N_1}_x \partial^{N_2}_\csi
        v( x, \csi)\right| \leq C_{N_1, N_2} \langle \csi \rangle^{m -
          \delta N_2}\ , \quad \forall \csi \in \R^d\,,
	\end{equation}
	where $\langle \xi \rangle := \left(1 + \| \xi\|^2
        \right)^{\frac{1}{2}}$. 
	\end{definition}
	\begin{definition}\label{def.op}
For {$0<\delta\leq 1$,} $m \in \R$, given a symbol $v \in \simb^m,$ we define the corresponding Weyl operator by 
	\begin{equation}\label{weyl.for.president}
	V  \psi(x) = \sum_{\csi \in \Z^d} \sum_{k \in \Z^d} \hat{v}_k\left(\csi + \frac{k}{2}\right) \hat{\psi}_\csi e^{\im (\csi+ k)\cdot x} \quad \forall \psi \in L^2(\T^d)\,.
	\end{equation}
	A linear operator $V$ is said to be a pseudo-differential
        operator of order $m$ if there exists $v \in \simb^m$ such
        that \eqref{weyl.for.president} holds; in such a case we write
        $\displaystyle{V = \Op(v) \in \OPS^m.}$
	\end{definition}
	\begin{remark}
		For any $m \in \R$, the space $S^m_\delta$  endowed with the
                sequence of seminorms given by the constants
                $C_{N_1, N_2}$ of \eqref{simb}, is a
                Fr\'{e}chet space. The same is true for the space
                $\OPS^m$ provided one defines the seminorms of an
                operator to be the
                seminorms of its symbol.
	\end{remark}

	\begin{remark}\label{lem.real.sad}
	It is well known that a pseudo-differential operator $V = \Op(v)$ is self-adjoint on $L^2(\T^d)$ if and only if the corresponding symbol $a$ is real valued.
	\end{remark}
	\begin{definition}\label{def.smooth}
          { If $\cF$ is a Fr\'echet (or a Banach) space, we denote
            by  $\timereg\left(\R; \cF\right)$ the space of the functions
            $V\in C^{\infty}\left(\R; \cF\right)$, $V=V(t)$ such that all the
            seminorms of $\partial^k_tV(t)$ are bounded uniformly in $t$.}
	\end{definition}
	For any $\sigma \geq 0$, we define the Sobolev space $H^\sigma = H^\sigma(\T^d)$ {as the closure} of $\cinfty(\T^d)$ with respect to the norm
	\begin{equation}\label{scherzavo}
	\| \psi\|_{\sigma}^2 = \sum_{\csi \in \Z^d} \langle \csi \rangle^{2 \sigma} |\hat{\psi}_\csi|^2\,.
	\end{equation}
	Furthermore, for any $\sigma_1, \sigma_2 \geq 0$, we define $\cB(H^{\sigma_1}; H^{\sigma_2})$ as the space of bounded linear operators from $H^{\sigma_1}$ to $H^{\sigma_2}$, and for any $A \in \cB(H^{\sigma_1}; H^{\sigma_2})$ we define the standard operator norm as 
	$$
	\| A\|_{\sigma_1, \sigma_2} = \sup_{u \in H^{\sigma_1}\,,\ \| u\|^{}_{\sigma_1} = 1} \| A u\|^{}_{\sigma_2}\,.
	$$
	If $\sigma_1 = \sigma_2,$ we simply write $\cB(H^{\sigma_1})$ instead of $\cB(H^{\sigma_1}; H^{\sigma_2}).$
\red{
	For future reference, we recall the Calderon-Vaillancourt Theorem:
	\begin{theorem}[Calderon-Vaillancourt]\label{teo.cal-va}
	  For any $m \in \R$ one has that
          if $A\in \OPS^{m}$ then $A\in \cB(H^\sigma;H^{\sigma-m})$ for all $\sigma \in \R$. Moreover, for
          time dependent operators, one has that $A(.)\in \calc^\infty_b(\R;\OPS^{m})$ implies $A(.)\in \calc^\infty_b(\R;\cB(H^\sigma;H^{\sigma-m}))$.
\end{theorem}
	}
	Given a time dependent family of self-adjoint operators $A(t)$,
	consider the initial value problem
	\begin{equation}\label{i.v.p}
	\im \partialt \psi(t) = A(t) \psi(t)\,, \quad \psi(s) = \psi\,.
	\end{equation}
	When the solution $\psi(t)$ exists globally in time, for any
        $t, s\in \R$ we denote by $\cU_{A}(t, s)$ the evolution
        operator mapping $\psi \in L^2(\T^d)$ to $\psi(t)\,.$

	We are now in position to state the main result of our paper.
	\begin{theorem}\label{crescono}
	Let $H$ be as in \eqref{main.eq} with $V \in \timereg\left(\R;
        OPS_1^{m}\right)$ \red{a family of self-adjoint pseudo-differential operators} \red{of order} $m<2$. {Then for any $\sigma\geq 0$ and
          for any initial datum $\psi \in H^\sigma$ there exists a
          unique global solution $\psi(t) := \cU(t, s)\psi \in
          H^\sigma$ of the initial value problem
          \begin{equation}
            \label{p.abs}
	\im \partialt \psi(t) = H(t) \psi(t)\,, \quad \psi(s) = \psi\,.
          \end{equation}
        }
	Furthermore, for any $\sigma>0$ and $\varep>0$ there exists a
        positive constant $K_{\sigma, \varep}$ such that for any $\psi
        \in H^\sigma$ one has
	\begin{equation}\label{growth.eq}
	\|\cU_H(t, s) \psi\|_\sigma \leq K_{\sigma, \varep} \langle t -s \rangle^{\varep}\| \psi\|_\sigma \quad \forall t, s \in \R\,.
	\end{equation}
	\end{theorem}

{The main example we have in mind is the Hamiltonian of a particle in a
time dependent electromagnetic field \red{on $\T^d$}, namely
\begin{equation}
  \label{em}
H(t)=\left(-\im\nabla +A(x,t)\right)^2+\Phi(x,t)\ ,
  \end{equation}}
   \red{with $A,\ \Phi \in \cC^\infty_b(\R; \cC^\infty(\T^d))$.}\\     
	The rest of the paper is devoted to the proof of Theorem \ref{crescono}.
	\section{Normal form}
	\subsection{Statement of the normal form result}
First we fix once for all real parameters $\delta$, $\ep,$ $\tau >0$
fulfilling
	\begin{equation}\label{legami.tra.parametri}
	0<\ep (\tau +1) < \delta< 1\,, \quad \tau \geq d-1\,, \quad \delta + d(d + \tau + 1)\ep < 1\,, \quad m<2\delta\, .
	\end{equation}
        We also put
        $$
\delta_*:=\delta + d(d + \tau + 1)\ep\ .
$$
\begin{definition}\label{def.nor}[Normal form operator]
	$\displaystyle{\NForm = \Op(\nform)\in \timereg\left(\R; \OPS^{m}\right)}$ is said to be \emph{in normal form} (with parameters $\delta, \ep, \tau$) if its symbol
	$$
	\begin{gathered}
	\nform(t, x, \csi) = \sum_{k \in \Z^d} \hat{\nform}_k(t, \csi) e^{\im k \cdot x}\,,
	\end{gathered}
	$$
	satisfies
	$$
	\hat{\nform}_k(t, \csi) \neq 0 \Longrightarrow |\scal{\csi
        }{k}| \leq \langle \csi  \rangle^{\delta}\|k\|^{-\tau}
        \textrm{ and } \|k\| \leq \langle \csi \rangle^{\ep}\, .
	$$
	\end{definition}
	{\begin{definition}\label{coniuge}
		Given two families of self-adjoint operators $H(t)$
                and $H^+(t)$ and a family of operators $U(t)$,
                unitary in $L^2$,
                we say that $U(t)$ conjugates $H(t)$ to $H^+(t)$ if the following holds: $\psi(t)$ solves the equation
		$$
		\im \partialt \psi(t) = H(t) \psi(t) \quad \forall t\in \R
		$$
		if and only if $\phi(t):= U(t) \psi(t)$ satisfies
		$$
		\im \partialt \phi(t) = H^+(t) \phi(t) \quad \forall t \in \R\,.
		$$
	\end{definition}
	With the above definitions, we can state our normal form result:
	}
	\begin{theorem}\label{norm.form}
	Let $H$ be as in the equation \eqref{main.eq}, with $V \in
        \timereg\left(\R;\OPS^m\right)$, {$m<2\delta<2$.} Assume
        that $V(t)$ is a family of self-adjoint operators and let { 
          $\rho : = \min\{2\delta - m, \delta\}$}. Then for any $N \in \N$
        there exists a time dependent family $U_N(t)$ of unitary (in
        $L^2$) maps {which conjugate the Hamiltonian
          \eqref{main.eq}\red{, in the sense of Definition \ref{coniuge},} to
	\begin{equation}\label{eq.in.forma}
	\widetilde{H}^{(N)}(t) + R^{(N)}(t)\,,
	\end{equation}
	}	
	where
        \begin{equation}
          \label{Nf.1}
\widetilde{H}^{(N)}(t) = -\Delta_{g} +
        \NForm^{(N)}(t)
        \end{equation}
and the following properties hold:
	\begin{enumerate}
	\item $\NForm^{(N)} \in \timereg\left(\R; \OPS^m\right)$ is in normal form, and $\NForm^{(N)}(t)$ is a family of self-adjoint operators 
	\item $R^{(N)} \in \timereg\left(\R;\OPS^{{m-\rho N}}\right)$, and the family $R^{(N)}(t)$ is self-adjoint
	\item For any $\sigma \geq 0,$ $U_N,\ U_N^{-1} \in L^\infty\left( \R; \cB(H^\sigma)\right)\,.$
	\end{enumerate}
	\end{theorem}

\subsection{Scheme of the proof of Theorem \ref{norm.form}}\label{a.parole}
Theorem \ref{norm.form} is a time dependent variant of Theorem 5.1 of
\cite{nonris} (see also Theorem 2.18 of \cite{risonante} and Theorem
4.3 of \cite{PS10}). More precisely, Theorem 5.1 of \cite{nonris}
provides the conjugation of an operator $H = -\Delta_{g} + V$ with a
\emph{time independent} $V$ of \red{order} smaller than $2$, to an operator
$\widetilde{H} + R$, where $\widetilde{H}$ is in normal form and $R$
is regularizing. {Here we adapt the construction performed therein to
  the time dependent case. The key point is that}, as summarized
below, time enters only as a parameter in the construction.
  {To see this, we recall that the construction of \cite{risonante} is
  iterative and we focus on the first step. We look for
a transformation of the form $ e^{\im G(t)},$} where $G \in
\timereg\left(\R; \OPS^{\eta}\right)$, {$\eta<\delta$} is
a family of self-adjoint operators. {Then, by Lemmas
  \ref{quantum.lie}, \ref{egorov}}, $e^{\im G(t)}$ conjugates the Hamiltonian
\eqref{main.eq} to $H^+(t)$ given by
		\begin{equation}\label{h.pl.a.pezzi}
		\begin{aligned}
		H^+(t) &= -\Delta_{g}  \\
		& - \im
		[-\Delta_{g}; G(t)] + V(t) \\
		&- \int_{0}^{1} e^{\im \tau G(t) }\partialt\left({G(t)}\right) e^{-\im \tau G(t)}\ d\tau  + R(t) \quad \forall t\in \R\,,
		\end{aligned}
		\end{equation}
		and $R \in \timereg\left(\R; \OPS^{m
                  +\eta-\delta}\right)$. If $\eta<
                  \min \{m, \delta\}$,  both the operators in the
                third line of \eqref{h.pl.a.pezzi} are remainder terms, in the sense that they are
                of order lower than $V(t)$. This is the reason why
                  time only plays the role of a parameter in the
                  normal form construction and the rest of the proof
                  is almost exactly as in \cite{nonris}. In order to be
                  self-contained, we give the details of the proof in Appendix \ref{sec.details}. 

	\section{Analysis of the normal form operator}\label{sez.struct.thm}
	{The starting point for the analysis of the dynamics of the
          normal form \eqref{Nf.1} is the structure theorem proved in
          \cite{risonante}. 
          To recall it we start by giving a few definitions.}
	
	\begin{definition}\label{def.cal.e}
		Let $E \subseteq \Z^d$. We call
		$$
		\cE := \overline{\Span \left\lbrace e^{\im \csi \cdot x}\ \left|\ \xi \in E \right.\right\rbrace} \subseteq L^2(\T^d)\,
		$$
		\emph{the subspace generated by $E$}.
	\end{definition}
	\begin{definition}
	Given a subspace $\cal E$ of $L^2(\T^d)$, we define $P_{\cal
          E}$ as the orthogonal projection on $\cal E$.
	\end{definition}
	\begin{definition}
		\label{mod}
		A subgroup $M$ of $\Z^d$ is called a module if
		$\Z^d\cap\Span_{\R}M=M$. Here and below, $\Span_{\R}M$ is the subspace generated
		by taking linear combinations with real coefficients of elements of
		$M$. {If $M \neq \{0\}$ and $M \neq \Z^d,$ we say that $M$ is a non trivial module.}
	\end{definition}
	Given a covector $\csi \in \Z^d$ and a module $M$, we denote
	$$
	\csi=\csi_M+ \csiort\ ,\quad \csi_M\in\Span_{\R}M\ ,\quad \csiort\in
	(\Span_{\R}M)^{\perp} \ ,.
	$$
%
	{The following result was proved in \cite{risonante}:}

\begin{theorem}\label{struct.resonant}[Corollary 5.11 of \cite{risonante}]
Let $\delta, \ep, \tau,\delta_*$ be parameters fulfilling 
\eqref{legami.tra.parametri}. For
        any modulus $M\subset\Z^d$ there exists a countable set ${\cal
          I}_M$ such that there exists a partition of $\Z^d$
$$ \Z^d = \bigcup_{M } \bigcup_{i \in {\cal I}_M} W_{M, i}\,,
$$ {where $M$ varies in the set of all the modules of $\Z^d$. Such a partition
          has the property  that any operator in normal form
          leaves invariant all the subspaces $\cW_{M,i}$ generated by
          $W_{M,i}$.  Furthermore, each set $W_{M, i}$ has
          finite cardinality, and the following properties hold:
\begin{enumerate}
\item $\forall i \in \cI_{\{0\}}$ the set $W_{\{0\}, i}$ has cardinality 1, namely $\sharp W_{\{0\}, i} = 1$.
\item If $M = \Z^d$, $\mort = \{0\}$ and $W_{\Z^d,\{0\}}$ has
  cardinality bounded by an integer $n_*$ which depends on the metric
  $g$ and on $d,\ \delta,\ \ep, {\tau}$ only
				\item
				If $M$ is a non trivial module, for any $i \in \mort$ one has: $\csi_{\mper} = (\csi')_{\mper}$  $\forall \csi, \csi' \in W_{M, i}$, and
				\begin{equation}\label{disuguaglianza.potente}
				\| \csi_M\| \leq K \langle \csi \rangle^{\delta_*} \quad \forall \csi \in W_{M, i}\,,
				\end{equation}
				where $K$ is a positive constant depending only on $g$ and on $d,\ \delta,\ \ep,
				{\tau}.$
			\end{enumerate}}
	\end{theorem}

\begin{remark}
	The sets $\{\cW_{M, i}\}_{M, i}$ are mutually orthogonal, both
        with respect to the $L^2$ and with respect to the
        $H^\sigma$ scalar products.
\end{remark}
We now prove {the following result. 
	\begin{lemma}\label{ti.piace.vincere.facile}
	There exists a positive constant $C$, depending only on $d,$ $\delta, \ep, \tau$ and $g,$ such that for any module $M$ and any $i \in \mort$, either
	\begin{equation}\label{piccolo}
	\max_{\csi \in W_{M, i}} \| \csi\| \leq C ,
	\end{equation}
	or  
	\begin{equation}\label{disuguaglianza.piu.potente}
	\max_{\csi \in W_{M, i}}\| \csi\| \leq 2 \min_{\csi \in W_{M, i}}\| \csi\| \,.
	\end{equation}
	\end{lemma}
	} {\begin{proof} If $M=\{0\}$, Item $1.$ of Theorem
    \ref{struct.resonant} implies that $W_{M, i}$ satisfies
    \eqref{disuguaglianza.piu.potente} $\forall i$, since it only
    contains a single point, whereas if $M = \Z^d$, Item $2.$ of
    Theorem \ref{struct.resonant} ensures that \eqref{piccolo}
    holds. Let $M$ be a non trivial module and let $i \in
    \mort$. First, using that $\langle \csi \rangle^2 = \|\csi_M\|^2 +
    \langle \csiort \rangle^2$, one has that
\begin{equation}\label{basso}
\min_{\csi \in W_{M, i}}\| \csi\| \geq \left\|\xi_{M^\bot}\right\| \ ,
\end{equation}
and by the equation \eqref{disuguaglianza.potente}, one gets
	\begin{equation}\label{bound.1}
		\|\csi_M\| < K^\prime \langle \csiort \rangle^{\delta_\ast}
	\end{equation}
for some $K' > 0$. Therefore one has 
$$
\left\|\xi_M\right\|\leq \sqrt{3}\left\|\xi_{M^\bot}\right\|
$$
provided {$\left\|\csiort\right\|>R$} for some constant $R > 0$ which depends on $K'$.
The latter inequality implies that
\begin{equation}\label{alto}
\left\|\xi\right\|^2=\left\|\xi_M\right\|^2+\left\|\xi_{M^\bot}\right\|^2\leq
4 \left\|\xi_{M^\bot}\right\|^2
\end{equation}
{and the estimate \eqref{disuguaglianza.piu.potente} follows.}
Note that $\left\|\xi_{M^\bot}\right\|\leq R$ implies
$\left\|\xi\right\|\leq C$ for some constant $C > 0$. The claimed statement has then been proved. 
	\end{proof}
	\section{Time evolution of Sobolev norms}

	Given $N \in \N,$ we start by providing upper bounds on the Sobolev norms of the solutions of the system
	\begin{equation}\label{normal form equation}
	\im \partialt \psi(t) = \widetilde{H}^{(N)} (t) \psi(t)\,,
	\end{equation}
	with $\widetilde{H}^{(N)}$ in normal form. 
        
	\begin{lemma}\label{lemma.flusso.blocchi}
	For any $\sigma \geq 0$ there exists a positive constant $K_\sigma$ such that, $\forall \psi \in H^\sigma$
	\begin{equation}\label{bounded.flow}
	\| \cU_{\widetilde H^{(N)}}(t, s)\psi\|_{\sigma} \leq K_\sigma \| \psi\|_{\sigma} \quad \forall t, s \in \R\,.
	\end{equation}
	\end{lemma}
	\begin{proof}
	{ First we prove the result in each
            block $\cW_{M,i}$. The result is obvious for blocks with uniformly
            bounded dimension. For the others we exploit the fact that
            they are dyadic namely they fulfill \eqref{disuguaglianza.piu.potente}. Precisely, we first remark that, by
            selfadjointness of $\widetilde H_{M,i}^{(N)}:=\widetilde
            H^{(N)}\big|_{\cW_{M,i}} $, the $L^2$
            norm is conserved under the flow $\cU_{\widetilde H^{(N)}_{M,i}}$. Then, denoting for
            simplicity $ \psi_{M,i}(t):=  \cU_{\widetilde H^{(N)}_{M,i}}(t, s)P_{\cW_{M,i}}\psi$ \red{and letting $c>0$ be such that $\displaystyle{\langle \csi \rangle \leq c \| \csi\|}$ $\forall \csi \neq 0,$} in each block one has  
            \begin{equation}\label{stima}
            \begin{aligned}
\left\|\psi_{M,i}(t) \right\|_\sigma & \red{\leq \left( \red{c} \max_{ \xi \in W_{M,i}}\left\|\xi\right\|\right)^\sigma\left\|\psi_{M,i}(t)
	\right\|_0}\\
	&\leq 
\left( \red{2 c} \min_{ \xi \in W_{M,i}}\left\|\xi\right\|\right)^\sigma\left\|\psi_{M,i}(t)
\right\|_0
\\
&= \left(\red{2 c} \min_{ \xi \in W_{M,i}}\left\|\xi\right\|\right)^\sigma\left\|\psi_{M,i}(0)
\right\|_0  
  \leq \red{(2 c)^{\sigma}} \left\|\psi_{M,i}(0)
\right\|_\sigma \ .
\end{aligned}
          \end{equation}
Then one has
		\begin{equation}\label{u.a.pezzi}
		\begin{aligned}
		\| \cU_{\widetilde{H}^{(N)}}(t, s) \psi \|^2_\sigma & = \| \cU_{\widetilde{H}^{(N)} }(t,s) \sum_{M, \red{i}} P_{\cW_{M, i}} \psi \|^2_\sigma \\
		&= \sum_{M, i} \| \psi_{M,i}(t) \|^2_\sigma\,.
		\end{aligned}
                \end{equation}
Using \eqref{stima} the result immediately follows. 
}		\end{proof}

\begin{proposition} \label{lemma.aggiungi.reg}
  Let ${N'} \in \N$ and \blu{let $H_0(t)$, $R(t)$
be time dependent families of self-adjoint operators with the
following properties
			\begin{enumerate}
				\item {There exists a positive constant $K_{N'}$} such that
				$$
				\left\|{\cU_{H_0}(t, s) \phi}\right\|_{{N'}} \leq K_{{N'}} \|\phi\|_{{N'}} \quad \forall t, s \in \R\,,\ \phi \in {H^{N'}}\,,
				$$
				\item $R = \Op(r) \in \timereg\left(\R; \OPS^{-{N'}}\right)$
			\end{enumerate}
                        Consider
                        $$
H(t)=H_0(t)+R(t)\ ;
                        $$} then $\forall\psi \in H^{N'}$ there exists
  a unique global solution $\psi(t) = \cU_H(t, s) \psi \in H^{N'}$ of
  the initial value problem \eqref{p.abs}, and the map $\psi : t
  \mapsto \psi(t)$ is in $\calc\left(\R; H^{N'}\right) \cap
  \calc^1\left(\R; H^{N' - 2}\right)\,.$ Furthermore there exists $K'>
  0$, depending only on $N'$, on $K_{{N'}}$ and on the family of
  seminorms of $R,$ such that for any $\psi \in {H^{N'}}$ one has
			\begin{equation}\label{flow.hn}
			\left\| \cU_H(t, s) \psi \right\|_{N'} \leq K^\prime  \langle t -s \rangle   \| \psi\|_{N'} \quad \forall t, s \in \R\,.
			\end{equation}
		\end{proposition}
                \begin{proof}\red{
We start by proving global well posedness. As usual we find solutions of the pseudo-PDE, on a time interval $[s - T, s + T]$
	$$
	\im \partial_t \psi =  ({H}_0 + R) \psi, \quad \psi(s) = \psi_0
	$$
	arguing by a standard fixed point argument on the map 
	$$
	\Phi(\psi) := {\cal U}_{H_0}(\blu{t,s}) \psi_0 + \int_{s}^{t} \cU_{H_0}(t, s^\prime) (-\im R(s^\prime))  \psi(s')\ d s^\prime\
	$$
	in a suitable ball ${B_{\blu{N'}}(\rho) := \{ \psi \in {\cal C}([s - T, s + T], H^{\blu{N'}}) : \| \psi \|_{L^\infty(\R, H^{\blu{N'}})} \leq \rho \}}$, where $\rho = \rho(\| \psi_0 \|_{\blu{N'}})$ and $T = T(\rho)$ have to be chosen appropriately. 
	Moreover, a Gronwall inequality implies that 
	$$
	\| \psi(t) \|_{\blu{N'}} \leq C_{\blu{N'}} {\rm exp}\Big( T \| R \|_{L^\infty(\R, {\cal B}(H^{\blu{N'}}))}\Big) \| \psi_0 \|_{\blu{N'}}\,,
	$$
	for a suitable positive constant $C_{\blu{N'}}.$ Hence $\psi(t)$ does never blow up in ${[s - T, s + T]}$ and hence it can be extended outside this interval. This implies that it is globally defined.}

We proceed proving estimate \eqref{flow.hn}. 
    Let
$\psi \in {H^{N'}}.$ By Duhamel formula
			\begin{equation}
			\cU_H(t, s) \psi = \cU_{H_0}(t, s) \psi + \int_{s}^{t} \cU_{H_0}(t, s^\prime) (-\im R(s^\prime)) \cU_{H}(s^\prime, s) \psi\ d s^\prime\,.
			\end{equation}
By \red{the Calderon Vaillencourt Theorem \ref{teo.cal-va}, one has} $R\in
                        \timereg\left(\R; \cB(H^{0};H^{N'})\right)$. So
                        since $H(t)$ is self-adjoint in $L^2$, 
one gets 
			\begin{align*}
			&\|\cU_H(t,s) \psi\|_{N'} \leq K_{N'} \| \psi \|_{N'} + \int_{s}^{t} K_{N'} \sobop{R(s^\prime)}{N'}{0} \| \cU_H(s^\prime, s)\psi\|_{0}\ d s^\prime\\
			&=  K_{N'} \| \psi \|_{N'} + \int_{s}^{t} K_{N'} \sobop{R(s^\prime)}{N'}{0} \| \psi\|_{0}\ d s^\prime\leq K_{N'} \| \psi \|_{N'} +  K_{N'} C \langle t -s\rangle \| \psi \|_0 \,,
                        \end{align*}
for some constant $C$. 
                \end{proof}
                
\begin{proof}[Proof of Theorem \ref{crescono}] \blu{Since $N$ is
    arbitrary in Theorem \ref{norm.form}, equation \eqref{flow.hn}
    holds for the normalized operator with $N' = \lfloor m - N \rho
    \rfloor$, so
    one gets
    \begin{equation}
      \label{flussi.nor}
	\| \cU_{\widetilde{H}^{(N)} + R^{(N)}}(t,s) \psi \|_{N^\prime}
        \leq K^\prime \langle t-s \rangle \|\psi \|_{N^\prime}
        \quad \forall t, s \in \R,\ \forall N\in\N\ .
    \end{equation}
By Theorem \ref{norm.form}, $H$ is unitarily equivalent to
$\widetilde{H}^{(N)} + R^{(N)}$ via a map $U_N$ such
that $\displaystyle{U_N, U_N^*\in L^\infty\left(\R;
  \cB(H^{N'})\right)}$, so we get
    \begin{equation}
      \label{flussi.nor.1}
	\| \cU_{H}(t,s) \psi \|_{N'}
        \leq K^\prime \langle t-s \rangle \|\psi \|_{N'}
        \quad \forall t, s \in \R,\ \forall N'\in\N\ .
    \end{equation}
We proceed now by interpolation: \red{given}
$\sigma>0$, take $N'>\sigma$ and $\red{\varepsilon}:=\sigma/N'\in(0,1)$, then one has
     	for any $t$ and $s \in \R$
     	$$
     	\begin{aligned}
     	\| \cU_H(t, s)\|_{\sigma, \sigma} &= \| \cU_H(t, s) u\|_{\varepsilon N', \varepsilon N'}\\
     	&\leq \left(\| \cU_H(t, s)\|_{N', N'}\right)^{\varepsilon} \left(\| \cU_H(t,s)\|_{0, 0}\right)^{1 - \red{\varepsilon}}\\
     	&\leq (K^\prime)^{\varepsilon} \langle t -s \rangle^{\varepsilon }\, .
     	\end{aligned}
     	$$
       Taking $N'$ large enough one can make $\varepsilon$ arbitrarily small, 
     	which gives the thesis.}
	\end{proof}

	\appendix
	
	{\section{Lie transform}}\label{lie.section}
{For completeness, in this appendix we recall the lemmas of
  \cite{growth} which are needed in the proof of Theorem \ref{norm.form}}
	\begin{lemma}\label{quantum.lie} [Lemma 3.1 of \cite{growth}]
		Let $H(t)$ and $G(t)$ be two smooth families of
                self-adjoint operators. \red{Recall Definition \ref{coniuge} of conjugated operators; then} one has that
                $e^{\im G(t)}$ conjugates
                $H$ to
		\begin{equation}
		H^+(t) = e^{\im G(t)} H(t) e^{-\im G(t)} - \int_{0}^{1} e^{\im \tau G(t)} \partialt(G(t)) e^{-\im \tau G(t)}\ d \tau \quad \forall t \in \R\,.
		\end{equation}
	\end{lemma}
 The following lemma ensures that
	$e^{\im G(t)} H(t) e^{-\im G(t)}$ is still a family of time dependent
	pseudo-differential operators provided $G(t)$ and $H(t)$
        are. It also gives an expansion in
	pseudo-differential operators of decreasing order:
	
	\begin{lemma}[\red{Lemmas 3.2 and 3.3} of \cite{growth}]\label{egorov}
		Let $G \in \timereg\left(\R; \OPS^\eta\right)$ with
                $\eta < \delta\,$ and suppose that $G(t)$ is a
                self-adjoint family of operators.
Given $A \in \timereg\left(\R; \OPS^m\right)$ for some $m \in \R$, \red{for any $\tau \in [-1, 1]$ and for any $t\in \R$ $e^{\im \tau G(t)}$ is a unitary operator, and
	\begin{equation}\label{exp.l.inf}
	e^{\im \tau G} \in L^\infty\left(\R; \cB(H^\sigma)\right) \quad \forall \sigma \geq 0\,.
	\end{equation}
	Furthermore,} for any $N \in \N$ one has $\displaystyle{e^{\im \tau G} A e^{-\im \tau G}\in \timereg\left(\R; \OPS^{m}\right),}$ with
		\begin{equation}
		e^{\im \tau G} A e^{-\im \tau G} = \red{\sum_{j = 0}^{N-1}} \frac{(\im \tau)^{j}\textrm{Ad}^{\ j}_{G}A}{j!} + \timereg\left(\R; \OPS^{m + N(\eta -\delta)}\right)\,,
		\end{equation}
		where $\{\textrm{Ad}^{\ j}_{G}A \}_{j \in \N}$ is defined by
		$$
		\textrm{Ad}^{\ 0}_{G}A = A\,, \quad \textrm{Ad}^{\ j+1}_{G}A =  [G, \textrm{Ad}^{\ j}_{G}A] \quad \forall j \geq 1\,.
		$$
	\end{lemma}

	\section{Proof of Theorem \ref{norm.form}}\label{sec.details}
	{In this appendix, again for sake of completeness, we
          show how to adapt the methods of \cite{nonris} to the time
          dependent case.}

{          The first point is to find a family of operators} $G(t) \in
          \OPS^{\eta}$ with $\eta < \min\{m, \delta\}$ s.t. the second
          line of Eq. \eqref{h.pl.a.pezzi} turns out to be in normal form. 
          To construct it, we work at the level of symbols, namely we set $G(t) = \Op(g(t))$ and $V(t)=\Op(v(t))$ and observe that
		\begin{equation}\label{i.want.to.be.in.normal.form}
		- \im
		[-\Delta_{g}; G(t)] + V(t) = \Op\left( \{\|\csi\|^2; g(t)\} + v(t)\right)\,.
		\end{equation}
		{Then we proceed by decomposing $v$ into normal
                  form, non-resonant and smoothing parts. To this aim,}
		consider an even smooth cutoff function $\chi: \R \rightarrow [0, 1]$ with the property that {$\chi(y) = 1$} for all $y$ with $|y| \leq
		\frac{1}{2}$ and {$\chi(y) = 0$} for all $y$ with $|y| \geq 1.$ 
		
			Given $\ep, \delta > 0$ and $\tau > d-1$ fulfilling \eqref{legami.tra.parametri}, define the following functions:
			\begin{align*}
			\chi_k(\csi) = \chi\left(\frac{2 \|k\|^\tau (\csi, k) }{\langle
				\csi \rangle^\delta}\right)\ , \quad  k &\in \Z^d\backslash
			\{0\},
			\\
			\tilde{\chi}_k(\csi) = \chi \left(\frac{\|k\|}{\langle \csi 
				\rangle^\ep}\right)\ , \quad k &\in \Z^d\backslash \{0\}\ .
			\end{align*}
			Correspondingly, given a symbol $w \in \timereg\left(\R^; S^{m}_\delta\right)$, we decompose
			it as follows:
			\begin{equation}\label{w a pezzi}
			w = \langle w \rangle +  w^{(\rm nr)} + w^{(\rm res)} + w^{(S)}\,,
			\end{equation}
			where 
			\begin{equation}\label{decomp} 
			\begin{gathered}
			\langle w \rangle (t, \csi) = \frac{1}{\mu(
                          \T^d)}\int_{\T^d} w (t, x, \csi)\ d \mu(x)\ ,
                        \\
                          w^{(\rm res)}(t, x, \csi) = \sum_{k \neq 0} \hat{w}_k(t, \csi) \chi_k(\csi) \tilde{\chi}_k(\csi) e^{\im k \cdot x}\,,\\
			w^{(\rm nr)}(t, x, \csi) = \sum_{k \neq 0} \hat{w}_k(t, \csi) \left(1 - \chi_k(\csi)\right) \tilde{\chi}_k(\csi) e^{\im k \cdot x}\,,\\
			w^{(S)}(t, x, \csi) = \sum_{k \neq 0} \hat{w}_k(t, \csi) \left(1 - \tilde{\chi}_k(\csi)\right) e^{\im k \cdot x}\,.
			\end{gathered}
			\end{equation}

		\begin{lemma}[Lemma 5.6 of \cite{nonris}]\label{lem.dec.ok}
			Let $w \in \timereg\left(\R; \simb^{m}\right)$ for some $m \in \R$; then $\langle w
			\rangle\,, w^{(\rm nr)}\,, w^{(\rm res)} \in \timereg\left(\R; \simb^{m}\right)$, and $w^{(S)} \in \timereg\left(\R; \simb^{-\infty}\right)$. Furthermore, if $w$ is real valued, also $w^{\res}, w^{\nr}, w^{(S)}$ and $\langle w \rangle$ are real valued. 
		\end{lemma}
		Then, in order to reduce to normal form the contribution of \eqref{i.want.to.be.in.normal.form}, one is left to solve the equation
		\begin{equation}\label{hom.sweet.hom}
		\{\| \csi\|^2; g\} + v^{\nr} = 0\,,
		\end{equation}
		which is dealt with in the following lemma:
		\begin{lemma}[Lemma 5.8 of \cite{nonris}]\label{lem.hom.eq}
			Let $w \in \timereg\left(\R; \simb^{m}\right)$ for some $m \in \R$; then
			\begin{equation}
			g(t, x, \csi) := \sum_{k \in \Z^d} \frac{\hat{w}^{\nr}_k(t, \csi)}{2 {\im}\scal{\csi}{k}} e^{\im k \cdot x}
			\end{equation}
			solves \eqref{hom.sweet.hom}, and $g \in
                        \timereg\left(\R; \simb^{m-\delta}\right)\,.$
                        Furthermore, if $w$ is real valued, then $g$ is real valued.
		\end{lemma}
		This leads to the proof of Theorem \ref{norm.form}:
		\begin{proof}[Proof of Theorem \ref{norm.form}]
			The result is obtained arguing by induction.
			When ${N=0,}$ $H(t) = H_0(t)$ is of the form
                        \eqref{eq.in.forma}, with $\NForm^{(0)} = 0$
                        and ${R}_0(t) = V(t)$, thus the zero-th step of the induction is immediately satisfied with $U_0 =\id$. Suppose that the thesis holds for some $N\in \N$: then one looks for a family of self-adjoint pseudo-differential operators
			$G_N(t)$ such that {$ e^{\im G_N(t)}$ conjugates $H_N(t) = \widetilde{H}^{(N)}(t) + R_N(t)$ satisfying properties $1$ -- $3$ of Theorem \ref{norm.form} to an operator $H_{N+1}(t)$ satisfying the same properties with $N$ replaced by $N+1$.}
			For all  $t \in \R$, let $r_N(t)$ be the
                        symbol of $R_N(t)$ and decompose $r_N(t)$
                        according to \eqref{w a pezzi}, namely
			$$
			r_N(t) = \langle r_N \rangle(t) + r_N^{\res}(t) + r_N^{\nr}(t) + r_N^{(S)}(t)\,.
			$$
			One then sets $G_N(t) = \Op(g_N(t)),$ where
                        $g_N$ is the solution of the homological
                        equation \eqref{hom.sweet.hom} with $w^{\nr} =
                        r_N^{\nr}$. By Lemma \ref{lem.hom.eq}, $G_N
                        \in \timereg\left(\R; {\OPS^{m -\rho N
                            -\delta}}\right)$ and $G_N(t)$ is
                        self-adjoint for all $t \in \R$. By Lemmas
                        \ref{quantum.lie}, \ref{egorov}, {one has that
                          $e^{\im G_N(t)}$ conjugates $H_N(t)$ \red{(in the sense of Definition \ref{coniuge})} to 
			$$
			\begin{aligned}
			H_{N+1}(t) &= -\Delta_g + Z_{N+1}(t) + R_{N+1}(t)\,,
			\end{aligned}
			$$}
with{
\begin{align}
  \label{rn}
  R_{N+1}(t)&=e^{\im G_N(t)} (-\Delta_g) e^{-\im G_N(t)} - (-\Delta_g) + \im [-\Delta_g, G_N(t)]\\
  &+ e^{\im G_N(t)} Z_N(t) e^{-\im G_N(t)}  - Z_N(t)\\
  & - \int_{0}^{1} e^{\im \tau G_N(t)} \partialt(G_N(t)) e^{-\im \tau G_N(t)}\ d \tau\\
  &+ e^{\im G_N(t)} R_N(t) e^{-\im G_N(t)}  - R_N(t)+ \Op(r_N^{(S)}(t))\,,
  \\
  Z_{N+1}(t) &=  Z_N(t) + \Op(\langle r_N \rangle(t)) + \Op(r_N^{\res}(t))\,,
\end{align}}
so that, in particular $R_{N+1}\in \timereg\left(\R;
\OPS^{m-{\rho}(N+1)}\right)\,$. Then the thesis immediately follows.
	\end{proof}


\def\cprime{$'$}

\end{document}